\documentclass[english]{amsart}

\usepackage{enumerate}
\usepackage{amssymb}
\usepackage{array}
\usepackage{graphicx}
\usepackage{lscape}
\usepackage{mathrsfs}
\usepackage{mathtools}
\usepackage{ mathdots }
\usepackage{textcomp}

\usepackage{color}
\usepackage[usenames,dvipsnames,svgnames,table]{xcolor}
\usepackage[curve,all]{xy}
\xyoption{matrix}
 \xyoption{curve}
 \xyoption{color}
 \xyoption{line}
\xyoption{arc}

\usepackage[shortlabels]{enumitem}
\setlist[enumerate]{leftmargin=7mm,topsep=0pt,itemsep=-1ex,partopsep=1ex,parsep=1ex,label=\rm{(\roman*)}}
\setlist[itemize]{label=\raisebox{0.25ex}{\tiny$\bullet$}}
\usepackage[backref=false, colorlinks, linktocpage, citecolor = blue, linkcolor = blue]{hyperref}

\theoremstyle{plain}
\newtheorem{theorem}{Theorem}[section]
\newtheorem*{theoremaux}{Theorem \theoremauxnum}
\gdef\theoremauxnum{1}

\newtheorem{proposition}[theorem]{Proposition}
\newtheorem*{propositionaux}{Proposition \propositionauxnum}
\gdef\propositionauxnum{1}

\newtheorem{lemma}[theorem]{Lemma}
\newtheorem*{lemmaaux}{Lemma \lemmaauxnum}
\gdef\lemmaauxnum{1}

\newtheorem{corollary}[theorem]{Corollary}

\theoremstyle{definition}
\newtheorem{definition}[theorem]{Definition}
\newtheorem{notation}[theorem]{Notation}
\newtheorem{example}[theorem]{Example}

\theoremstyle{remark}
\newtheorem{remark}[theorem]{Remark}

\newcommand{\leftexp}[2]{{\vphantom{#2}}^{#1}{#2}}

\newcommand{\incl}[1][r]{\ar@<-0.2pc>@{^(-}[#1] \ar@<+0.2pc>@{-}[#1]}

\newcommand{\hs}{\kern 0.8pt}

\newcommand{\Spec}{\mathrm{Spec}}
\newcommand{\inn}{\mathrm{inn}}

\newcommand{\V}{\mathcal{V}}

\renewcommand{\gg}{\mathfrak{g}}
\newcommand{\h}{\mathfrak{h}}

\newcommand{\iso}{\simeq}

\newcommand{\C}{\mathbb{C}}

\newcommand{\Z}{\mathbb{Z}}
\renewcommand{\l}{\mathfrak{l}}
\newcommand{\R}{\mathbb{R}}
\newcommand{\G}{\mathbb{G}}

\DeclareMathOperator{\SL}{SL}

\DeclareMathOperator{\SU}{SU}
\DeclareMathOperator{\Sp}{Sp}
\DeclareMathOperator{\Aut}{Aut}

\DeclareMathOperator{\Gal}{Gal}
\DeclareMathOperator{\GL}{GL}

\title[Real structures on symmetric spaces]{Real structures on symmetric spaces}

\author[Lucy Moser-Jauslin and Ronan Terpereau]{Lucy Moser-Jauslin and Ronan Terpereau}
\thanks{The second-named author is supported by the Project FIBALGA ANR-18-CE40-0003-01.
This work received partial support from the French "Investissements d\textquoteright Avenir" program and from project ISITE-BFC (contract ANR-lS-IDEX-OOOB).
The IMB receives support from the EIPHI Graduate School (contract ANR-17-EURE-0002)}

\address{Institut de Math\'{e}matiques de Bourgogne, UMR 5584 CNRS, Universit\'{e} Bourgogne Franche-Comt\'{e}, F-21000 Dijon, France}
\email{lucy.moser-jauslin@u-bourgogne.fr}

\address{Institut de Math\'{e}matiques de Bourgogne, UMR 5584 CNRS, Universit\'{e} Bourgogne Franche-Comt\'{e}, F-21000 Dijon, France}
\email{ronan.terpereau@u-bourgogne.fr}



\def\ga{\,^\gamma\hskip-1pt}

\begin{document}

\begin{abstract}
We obtain a necessary and sufficient condition for the existence of  equivariant real structures on complex symmetric spaces for semisimple algebraic groups and discuss how to determine the number of equivalence classes for such structures. 
\end{abstract}

\maketitle

\tableofcontents

\section*{Introduction}

A (complex algebraic) \emph{symmetric space} is a complex algebraic $G$-variety $X=G/H$, where $G$ is a complex reductive algebraic group, $\theta \in \Aut_\C(G)$ is a non-trivial group involution, and $H \subseteq G$ is a subgroup such that $(G^\theta)^0 \subseteq H \subseteq N_G(G^\theta)$. The historical motivation for the study of symmetric spaces comes from the \emph{Riemannian symmetric spaces} (see \cite{Hel78} for an exposition); those arise in a wide range of situations in both mathematics and physics, and local models are given by the real loci of certain (complex algebraic) symmetric spaces. Therefore, given a symmetric space, it is natural to ask whether it admits \emph{equivariant real structures} (see \S~\ref{subsec:equiv real structures}). The present note aims at providing an answer to this question. In this article, we restrict to the case where $G$ is semisimple  and simply-connected (see Rk.~\ref{rk:Gsemsimple}); in particular, by \cite[\S~8]{Ste68} this implies that $G^\theta$ is connected, i.e. that $(G^\theta)^0=G^\theta$.

A homogeneous space $G/H$ is \emph{spherical} if a Borel subgroup of $G$  acts with an open dense orbit; see \cite{Kno91,Tim11} for an exposition of the theory of spherical homogeneous spaces and their equivariant embeddings. Spherical homogeneous spaces are classified in terms of combinatorial data called \emph{homogeneous spherical data} \cite[\S~30.11]{Tim11}. By Vust \cite{Vus74} symmetric spaces are spherical, thus symmetric spaces are also classified by the homogeneous spherical data. However, these data are quite complicated to handle and also, those corresponding to symmetric spaces have no particular features to distinguish them from those corresponding to non-symmetric spaces. Using these data, a criterion for the existence of equivariant real structures on general spherical homogeneous spaces was obtained by Borovoi and Gagliardi in \cite{BG} (generalizing results of Akhiezer and Cupit-Foutou in \cite{ACF14,Akh15,CF15}). However, this criterion can be difficult to apply in specific cases. In particular, in the case where the spherical homogeneous space is a symmetric space, the involution $\theta$ is not used. 
The leading goal of this article is  to  obtain an independent practical criterion  using the involution $\theta$.

Our main result in this note is the following (Th.~\ref{th:main result existence} and Prop.~\ref{prop:number of structures}):

\begin{theorem} \label{th: main th}
Let $(G,\sigma)$ be a complex simply-connected semisimple algebraic group with a real group structure. 
Let $\theta$ be a (non-trivial regular) group involution on $G$, and let $G^\theta \subseteq H \subseteq N_G(G^\theta)$ be a symmetric subgroup. Then there exists a $(G,\sigma)$-equivariant real structure on the symmetric space $G/H$ if and only if the following holds:
\begin{itemize}[leftmargin=4mm]
\item the involutions $\sigma \circ \theta \circ \sigma$ and $\theta$ are conjugate by an inner automorphism;
\item the $\Z/2\Z$-action on $N_G(G^\theta)/G^\theta$ induced by $\sigma$ (see Def.~\ref{def:Gamma-action}) stabilizes $H/G^\theta$; and
  \item $\Delta_H(\sigma)=0$, where $\Delta_H$ is the map defined by \eqref{eq:map Delta} at the end of \S~\ref{subsec1}.

\end{itemize}
Moreover, if such a structure exists, then there are exactly $2^n$ equivalence classes of $(G,\sigma)$-equivariant real structures on $G/H$, where $n$ is a non-negative integer than can be calculated explicitly (see \S~\ref{sec:number of eq classes} for details).
\end{theorem}

\begin{remark} \label{rk:Gsemsimple}
The fact that $G$ is assumed to be semisimple, and not just reductive, is crucial for Prop.~\ref{prop:H abelian and finite}, Prop.~\ref{prop:fixed locus of an involution} (see also Rk.~\ref{rk:G not reductive}), and Cor.~\ref{cor:H conjugate for symmetric subgroups}. On the other hand, one can always replace $G$ by its universal covering space to reduce to the case where it is simply-connected.
\end{remark}

\begin{remark}
The fact that symmetric spaces are spherical is used in the proof of Prop. ~\ref{prop:H abelian and finite}, to say that $N_G(H)/H$ is an abelian group, and in the proof of Cor.~\ref{cor:H conjugate for symmetric subgroups}, to apply Prop.~\ref{prop:HcongugateH'} (which is the only result where some knowledge on the theory of equivariant embeddings for spherical homogeneous spaces is required). 
\end{remark}
\begin{remark}
Let $X=G/H$ be a symmetric space with a ($G,\sigma$)-equivariant real structure $\mu$ such that $X(\C)^\mu$ is non-empty. Then $G(\C)^\sigma$ acts on $X(\C)^\mu$ with finitely many orbits and a combinatorial description of these orbits using Galois cohomology is provided in \cite{CFT18} (see also \cite[Chp.~6]{BJ06}).
\end{remark}

In \cite{MJT18} we studied the equivariant real structures on \emph{horospherical varieties} which are another class of spherical varieties. The main result \cite[Th.~0.1]{MJT18} regarding the existence of equivariant real structures on horospherical homogeneous spaces  is quite similar to Th.~\ref{th: main th} but the case of horospherical homogeneous spaces differs greatly from the case of symmetric spaces for the following reasons:
\begin{itemize}[leftmargin=4mm]
\item The homogeneous spherical data corresponding to horospherical homogeneous spaces are easy to discriminate and take a very simple form (see \cite[\S~3.1]{MJT18} for details) contrary to the case of symmetric spaces.
\item The group $\Aut_\C^G(G/H)\iso N_G(H)/H$, which plays a key role when counting the number of equivalence classes of equivariant real structures on $G/H$, is a torus for horospherical homogeneous spaces while it is a finite abelian group for symmetric spaces (Prop.~\ref{prop:H abelian and finite}).
\item In both cases, an equivariant real structure on $G/H$ extends to a $G$-equivariant embedding $G/H \hookrightarrow X$ if and only if the corresponding \emph{colored fan} is stable for the induced action of the Galois group $\Gamma=\Gal(\C/\R)$ (see \cite{Hur11,Wed18}), but in the horospherical case the quotient $X/\Gamma$ is always an algebraic variety while in the symmetric case it can be an algebraic space. Therefore the question of the existence of real forms for symmetric varieties is subtler than for horospherical varieties, and that is the reason why in this note we only consider the homogeneous case.
\end{itemize}

\begin{remark}
A homogeneous space $G/H$ is \emph{horosymmetric} if it is a homogeneous fibration over a flag variety $G/P$, whose fibers are symmetric spaces. This class of spherical homogeneous spaces, which contains both symmetric spaces and horospherical homogeneous spaces, was introduced by Delcroix in \cite{Del20}. It would be interesting to determine a nice criterion for the existence of equivariant real structures on horosymmetric spaces from Th.~\ref{th: main th} and \cite[Th.~0.1]{MJT18}.
\end{remark}


In \S~\ref{sec:gen back on rea structures} we recall some definitions and results on real group structures, equivariant real structures, and symmetric spaces. Then in \S~\ref{sec:existence} we prove the necessary and sufficient condition of Th.~\ref{th: main th} for the existence of equivariant real structures on symmetric spaces (Th.~\ref{th:main result existence}). Finally, in \S~\ref{sec:number of eq classes} we show how to determine  the number of equivalence classes for such structures (Prop.~\ref{prop:number of structures}).

\bigskip

\noindent \textbf{Convention.}
In this article we work over the field of real numbers $\R$ and over the field of complex numbers $\C$.
We denote by $\Gamma$ the Galois group $\Gal(\C/\R)=\{1,\gamma\} \iso \Z/2\Z$.
\textbf{We will always denote by $G$ a complex (connected) simply-connected semisimple algebraic group.} 
We refer the reader to \cite{Hum75} for the standard background on algebraic groups.

\section{General background on real structures and symmetric spaces} \label{sec:gen back on rea structures}
This first section is a short recollection of general results on real structures detailed in \cite[\S\S~1-2]{MJT18} and on symmetric spaces.

\subsection{Real group structures} \label{subsec1}

\begin{definition} 
A \emph{real group structure} on $G$ is a scheme involution $\sigma\colon G \to G$ which makes the diagram 
\[\xymatrix@R=4mm@C=2cm{
    G \ar[rr]^{\sigma} \ar[d]  && G \ar[d] \\
    \Spec(\C)  \ar[rr]^{\Spec(z \mapsto \overline{z})} && \Spec(\C)  
  }\]
commute, and such that the induced morphism $\gamma_*G \to G$ is an isomorphism of complex algebraic groups, where $\gamma_*G \to \Spec(\C)$ is the base change of $G \to \Spec(\C)$ along the morphism at the bottom of the square (see \cite[\S~2.5]{BG}).  
  
  Two real group structures $\sigma$ and $\sigma'$ on $G$ are \emph{equivalent} if there exists a (regular) group automorphism $\varphi \in \Aut_{gr}(G$) such that $\sigma'=\varphi \circ \sigma \circ \varphi^{-1}$.
\end{definition}

\begin{remark}
The \emph{real locus} $G(\C)^\sigma$ of $(G,\sigma)$ is a real Lie group.
\end{remark}

The next result shows that the classification of the real group structures on complex simply-connected semisimple algebraic groups reduces to the simple groups case.

\begin{lemma} \label{lem:easy_lemma_reduction}  \emph{(\cite[Lem.~1.7]{MJT18})}
Let $G=\prod_{i \in I} G_i$, where the $G_i$ are the simple factors of $G$, and let $\sigma$ be a real group structure on $G$. 
Then, for a given $i \in I$, we have the following possibilities:
\begin{enumerate}
\item $\sigma(G_i)=G_i$ and $\sigma_{|G_i}$ is a real group structure on $G_i$; or
\item  there exists $j \neq i$ such that $\sigma(G_i)=G_j$, then $G_i \iso G_j$ and $\sigma_{| G_i \times G_j}$ is equivalent to $(g_1,g_2) \mapsto (\sigma_0(g_2),\sigma_0(g_1))$, where $\sigma_0$ is any real group structure on $G_i \iso G_j$.
\end{enumerate}
\end{lemma}

The real group structures on complex simply-connected simple algebraic groups are well-known (a table can be found in \cite[Appendix~A]{MJT18}); they correspond to real Lie algebra structures on complex simple Lie algebras (see \cite[\S~V\!I.10]{Kna02} for the classification of those in terms of diagrams).

\begin{definition} \label{def:qs real str}
 If there exists a Borel subgroup $B \subseteq G$ such that $\sigma(B)=B$, then $\sigma$ is called \emph{quasi-split}. For $c \in G$ we denote by $\inn_c$ the inner automorphism of $G$ defined by
 \vspace{-2mm}
\[ \inn_c: G \to G,\ g \mapsto cgc^{-1}.\]
If $\sigma_1$ and $\sigma_2$ are two real group structures on $G$ such that $\sigma_2 = \inn_c \circ \sigma_1$, for some $c \in G$, then $\sigma_2$ is called an \emph{inner twist} of $\sigma_1$.
\end{definition}

A quasi-split real group structure always preserves some maximal torus $T \subseteq B$. 
Moreover, any real group structure on $G$ is equal to the inner twist of a quasi-split real group structure unique up to equivalence \cite[Prop. 7.2.12]{Con14}.

Note that if $\sigma$ is equivalent to a quasi-split real structure $\sigma_{qs}$, then $\sigma$ is also quasi-split. However, the Borel subgroup which is preserved by $\sigma$ is not in general the same as the one preserved by $\sigma_{qs}$.

\begin{example} \label{ex1}
Let $n \geq 2$. There are exactly two inequivalent quasi-split real group structures on $\SL_{2n}$. The first one is  the split real group structure $\sigma_{sp}\colon g \mapsto \overline{g}$, whose real locus is $\SL_{2n}(\R)$, and the second is the real group structure defined by $\sigma_{qs}: g \mapsto K_{n,n} \,{}^{t}(\overline{g})^{-1}\, {}^{t}K_{n,n}$ with $K_{n,n}=\begin{bmatrix}
0& C_n \\ -C_n & 0
\end{bmatrix}$, where  $C_n \in \GL_n(\C)$ is given by $C_n=\begin{bmatrix}
0 & 0 & 1 \\0 & \iddots & 0 \\1 & 0 & 0\end{bmatrix}$. Let us note that ${}^{t}K_{n,n}=K_{n,n}^{-1}=-K_{n,n}$. The real locus of $\sigma_{qs}$ is the real Lie group  $\SU(n,n,\R)$. 
\end{example}

Recall that we denote $\Gamma=\Gal(\C/\R)=\{1,\gamma\} \iso \Z/2\Z$. Let $\sigma_{qs}$ be a quasi-split real group structure on $G$.
We consider the short exact sequence of $\Gamma$-groups
\[1\to Z(G)\to G\to G/Z(G)\to 1,\]
where the $\Gamma$-action is induced by $\sigma_{qs}$. In other words, the element $\gamma\in\Gamma$ acts on $G$ and $Z(G$) by $\sigma_{qs}$, and on $G/Z(G$) by the induced real group structure.
For $\mathcal{G}$ a $\Gamma$-group, let $H^1(\Gamma,\mathcal{G})$ denote the first Galois cohomology pointed set, and for $\mathcal{A}$ an abelian $\Gamma$-group, let $H^2(\Gamma,\mathcal{A})$ denote the second Galois cohomology group; see \cite{Ser02} for more details on Galois cohomology. 
There is a connecting map 
\begin{equation*}
\delta\colon H^1(\Gamma, G/Z(G))\to H^2(\Gamma,Z(G)).
\end{equation*}
Tables where the map $\delta$ is calculated when $G$ is simple can be found in \cite[Appendix~A]{MJT18}. 

\textbf{Let now $H$ be a subgroup of $G$ such that $\sigma_{qs}(H)=H$ and $N_G(H)/H$ is abelian.}
Then  $\sigma_{qs}$ induces a real group structure on $N_G(H)/H$, namely $\widehat{\sigma_{qs}}(nH)=\sigma_{qs}(n)H$, and we can consider the second cohomology group $H^2(\Gamma, N_G(H)/H)$. Since it is abelian, we will use the additive notation for this cohomology group, with the neutral element denoted by $0$. The natural homomorphism $\chi_H:Z(G)\to N_G(H)/H$, induced by the inclusion $Z(G) \to N_G(H$),
yields a homomorphism between the second cohomology groups
\[ \chi_H^*:H^2(\Gamma, Z(G)) \to H^2(\Gamma, N_G(H)/H).\]
In the rest of this article we will denote the composed map $\chi_H^* \circ \delta$ by
\begin{equation} \label{eq:map Delta} \tag{\textasteriskcentered}
\Delta_H: H^1(\Gamma, G/Z(G))  \to H^2(\Gamma, N_G(H)/H).
 \end{equation}
 We recall that $H^1(\Gamma,G/Z(G)) \iso \{c \in G \ | \ c \sigma_{qs}(c) \in Z(G)\}/\equiv$, where $c \equiv c'$ if $c^{-1}b^{-1}c' \sigma_{qs}(b) \in Z(G)$ for some $b \in G$. 

\begin{notation}
If $\sigma$ is a real group structure equivalent to $\inn_c \circ \sigma_{qs}$, then we will write $\Delta_H(\sigma)$ for the element $\Delta_H([c])$ of $H^2((\Gamma, N_G(H)/H)$. (Note that if $\inn_c\circ\sigma_{qs}=\inn_{c'}\circ\sigma_{qs}$, then the classes $[c]$ and $[c']$ are equal.)
\end{notation}
 
The element $\Delta_H(\sigma)  \in H^2(\Gamma, N_G(H)/H)$ is a \emph{cohomological invariant} that plays a key role in our criterion for the existence of equivariant real structures on symmetric spaces; see Th.~\ref{th: main th}.

\subsection{Equivariant real structures} \label{subsec:equiv real structures}
Let $\sigma$ be a real group structure on $G$.

\begin{definition}
A ($G,\sigma$)-\emph{equivariant real structure} on a $G$-variety $X$ is a scheme involution on $X$  such that the diagram 
\[
\xymatrix@R=4mm@C=2cm{
    X \ar[rr]^{\mu} \ar[d]  && X \ar[d] \\
    \Spec(\C)  \ar[rr]^{\Spec(z \mapsto \overline{z})} && \Spec(\C)  
  }
  \]
commutes, and such that 
\[ \forall g \in G, \; \forall x \in X, \;\; \mu(g \cdot x)=\sigma(g) \cdot \mu(x).\]

Two equivariant real structures $\mu$ and $\mu'$ on a ($G,\sigma$)-variety $X$ are \emph{equivalent} if there exists a $G$-automorphism $\varphi \in \Aut_\C^G(X)$ such that $\mu'=\varphi \circ \mu\circ \varphi^{-1}$.
\end{definition}

\begin{lemma} \label{lem: two conditions} \emph{(\cite[Lem.~2.4]{MJT18})}
Let  $X=G/H$ be a homogeneous space. 
Then $X$ has a $(G,\sigma)$-equivariant real structure if and only if there exists $g \in G$ such that these two conditions hold:
\begin{enumerate}[$(1)$]
\item \label{eq: sigma compatible}
\emph{$(G,\sigma)$-compatibility condition:}  $\sigma(H)=gHg^{-1}$
\item \label{eq: involution}
\emph{involution condition:}\hspace{15mm} $\sigma(g)g \in H$
\end{enumerate}
{in which case such a structure $\mu$ on $X$ is given by $\mu(kH)=\sigma(k)gH$ for all $k\in G$.}
\end{lemma}

\begin{remark} \label{rk:N(H)=H}
The first condition implies that $\sigma(g)g \in N_G(H)$, and so if $N_G(H)=H$ the second condition is automatically fulfilled.
\end{remark}

\begin{remark} \label{rk:conjugate} 
If $H'$ is conjugate to $H$, then $G/H$ has a ($G,\sigma$)-equivariant real structure if and only if $G/H'$ has a ($G,\sigma$)-equivariant real structure.
\end{remark}

\begin{proposition} \label{prop:coho condition} \emph{(\cite[Prop.~2.8]{MJT18}, see also \cite[Th.1.6]{BG})}\\
Let $\sigma=\inn_c \circ \sigma_{qs}$, where $c \in G$ and $\sigma_{qs}$ is a quasi-split real groups structure on $G$. Assume that $N_G(H)/H$ is abelian and that $\sigma_{qs}(H)=H$. Then
\begin{enumerate}
\item $G/H$ has a $(G,\sigma_{qs})$-equivariant real structure; and
\item\label{item:ii} $G/H$ has a $(G,\sigma)$-equivariant real structure if and only if $\Delta_H(\sigma)=0$,
\end{enumerate}
where $\Delta_H$ is the map defined by \eqref{eq:map Delta} at the end of  \S~\ref{subsec1}.
\end{proposition}

\begin{remark} Note that, since $\sigma$ defines a real group structure on $G$, we have that $c\sigma_{qs}(c)\in Z(G)$.  
Moreover, condition \ref{item:ii} of Proposition~\ref{prop:coho condition} is equivalent to the following condition:
\begin{itemize}
\item[(ii')] $G/H$ has a $(G,\sigma)$-equivariant real structure if and only if there exists $n\in N_G(H)$ such that $n\sigma_{qs}(n)\in c\sigma_{qs}(c)H$.
\end{itemize}
Thus, for example, if $H$ contains the center $Z(G)$, then this condition is trivially satisfied.
\end{remark}

\subsection{Symmetric spaces}
In this section we recall some basic facts on symmetric spaces that we will need in the following. The interested reader is referred to \cite[\S~26]{Tim11} for a detailed survey on symmetric spaces.

\begin{definition}
A subgroup $H \subseteq G$ is \emph{symmetric} if there exists a (non-trivial regular) group involution $\theta$ on $G$ such that $G^\theta \subseteq H \subseteq N_G(G^\theta)$. A homogeneous space $G/H$ is \emph{symmetric} if $H$ is a symmetric subgroup of $G$.
\end{definition}

\begin{example}
The group $G$ itself can be viewed as a symmetric space for the action of $G \times G$ by left and right multiplication. Indeed, $G \iso (G \times G)/H$, where $H=(G \times G)^\theta$ with $\theta(g_1,g_2)=(g_2,g_1)$.
\end{example}

\begin{example} \label{ex2}{
Let $n \geq 2$ and let $G=\SL_{2n}$. Then $\theta : g \mapsto J\,  ({}^{t}g^{-1})\, {}^{t}J$, with $J=\begin{bmatrix}
0 & I_n \\ -I_n & 0
\end{bmatrix}$ and ${}^{t}J=J^{-1}=-J$, is a group involution on $G$. We have $G^\theta=\Sp_{2n}$ and $N_G(G^\theta)=\left \langle Z(G), G^\theta \right \rangle$, and thus $N_G(G^\theta)/G^\theta \iso Z(G)/(Z(G) \cap G^\theta) \iso \Z/n\Z$.}
\end{example}

As for real group structures (see Lem.~\ref{lem:easy_lemma_reduction}), the next result shows that the classification of (regular) group involutions on simply-connected semisimple algebraic groups reduces to the case of simple groups.

\begin{lemma} \label{lem:easy_lemma_reduction2}  
Let $\theta$ be a group involution on $G=\prod_{i \in I} G_i$, where the $G_i$ are the simple factors of $G$. 
Then, for a given $i \in I$, we have the following possibilities:
\begin{enumerate}
\item \label{item: Gi stable} $\theta(G_i)=G_i$ and $\theta_{|G_i}$ is a group involution on $G_i$; or
\item there exists $j \neq i$ such that $\theta(G_i)=G_j$, then $G_i \iso G_j$ and $\theta_{| G_i \times G_j}$ is conjugate to $(g_1,g_2) \mapsto (g_2,g_1)$.
\end{enumerate}
\end{lemma}

\begin{proof}
We use the fact that the factors $G_i$ are the unique simple normal subgroups of $G$ (see \cite[Th.~5.1.19] {Con14}). In particular, any group automorphism of $G$ permutes the factors. Since $\theta$ is a group involution,
 either $\theta(G_i)=G_i$ and we get \ref{item: Gi stable}, or $\theta(G_i)=G_j$ for some $j \neq i$. In the second case,  $G_i$  and $G_j$ are then isomorphic. 
 Therefore
 $G_i \times G_j \iso H \times H$, for some simply-connected simple algebraic group $H$, and $\theta_{|G_i \times G_j}$ identifies with $
 \theta_{H \times H}: (h_1,h_2) \mapsto (\psi(h_2),\psi^{-1}(h_1))$ for some group automorphism $\psi$ on $H$. But
 then it suffices to conjugate $\theta_{H \times H}$ with the group automorphism defined by  $(h_1,h_2) \mapsto (\psi(h_2),h_1)$ to get the involution $(g_1,g_2) \mapsto (g_2,g_1)$.
\end{proof}

Conjugacy classes of (regular) group involutions on simply-connected simple algebraic groups can be classified by using either Kac diagrams or Satake diagrams; see \cite[\S~26.5]{Tim11} for more details on these classifications and \cite[Table~26.3]{Tim11} for the list of conjugacy classes of (regular) group involutions on simply-connected simple algebraic groups.

\begin{proposition} \label{prop:H abelian and finite}
For any symmetric subgroup $G^\theta \subseteq H \subseteq N_G(G^\theta)$, the quotient group $N_G(H)/H$ is abelian and finite; in particular, $G^\theta=H^0=N_G(H)^0=N_G(G^\theta)^0$.
\end{proposition}

\begin{proof}
Symmetric subgroups are spherical (see \cite{Vus74} or \cite[Th.~26.14]{Tim11}), and if $H$ is a spherical subgroup of $G$, then $N_G(H)/H$ is abelian (see \cite[\S~5.2]{BP87} or \cite[Th.~6.1]{Kno91}). We already mentioned the connectedness of $G^\theta$ in the introduction (as $G$ is semisimple and simply-connected), and the finiteness of $N_G(H)/H$ then follows from the work of De Concini and Procesi in \cite[\S~1.7]{DCP83} (see also \cite[\S~2.2]{Vus90}).   
\end{proof}

\section{Existence of equivariant real structures on symmetric spaces} \label{sec:existence}
In this section we will always denote by $\theta$ a (regular non-trivial) group involution on $G$, by $\sigma_{qs}$ a quasi-split real group structure on $G$, and by $\sigma=\inn_c \circ \sigma_{qs}$ a real group structure on $G$ obtained as an inner twist of $\sigma_{qs}$. 

\begin{notation}
To simplify the notation we will denote $\leftexp{\psi}\theta=\psi \circ \theta \circ \psi^{-1}$ for any (regular or antiregular) group automorphism $\psi$. Also, if $\theta_1$ and $\theta_2$ are two group involutions on $G$, we will write $\theta_1\sim\theta_2$ when they are conjugate by an inner automorphism of $G$.
\end{notation}

Example~\ref{ex2.2} shows that the combinatorial invariants of the conjugacy class of $\theta$ (such as  Kac diagrams or Satake diagrams) are too coarse to determine the existence of a $(G,\sigma)$-equivariant real structure on the symmetric space $G/G^\theta$.

\begin{example}\label{ex2.2}
Let $G=\SL_n^{\times 3}$ with $n \geq 2$, and let $\sigma\colon (g_1,g_2,g_3)\mapsto (\overline{g_2}, \overline{g_1},\leftexp{t}{\overline{g_3}^{-1}})$ be a real group structure on $G$. We give an example of two group involutions $\theta$ and $\theta'$ that are conjugate (by an outer automorphism of $G$) such that $G/G^\theta$ admits a $(G,\sigma)$-equivariant real structure but $G/G^{\theta'}$ does not. 
Let $\theta\colon (g_1,g_2,g_3) \mapsto (g_2,g_1,\leftexp{t}{g}_{3}^{-1})$, let $\psi\colon(g_1,g_2,g_3) \mapsto (g_3,g_2,g_1)$, and let  $\theta'=\leftexp{\psi}\theta$. Then $\sigma(G^\theta)=G^\theta$ while $\sigma(G^{\theta'})$ is not conjugate to $G^{\theta'}$ in $G$, and we conclude with Lem.~\ref{lem: two conditions}.
\end{example}

Therefore, a criterion for the existence of a $(G,\sigma)$-equivariant real structure on $G/G^\theta$ should depend on $\theta$ up to a  conjugate by an \textbf{inner} automorphism of $G$, which is indeed the case in Th.~\ref{th: main th}. 

\smallskip

The next result is well-known to specialists but we give a proof because of a lack of suitable reference. We thank Michael Bulois for indicating us the sketch of the proof. 

\begin{proposition} \label{prop:fixed locus of an involution}
The group involution $\theta$ on the simply-connected semisimple algebraic group $G$ is uniquely determined by its fixed locus $G^\theta$.
\end{proposition}

\begin{proof}
As group involutions on simply-connected semisimple algebraic groups correspond to Lie algebra involutions on semisimple Lie algebras (see \cite[\S~4.3.4]{Pro07}), it suffices to prove that if {$\Theta(=D_e\theta)$} is a Lie algebra involution on $\gg=Lie(G)$, then $\gg^\Theta$ determines $\Theta$. 

Since $\gg$ is semisimple, it identifies with a direct sum of simple Lie algebras $\gg=\bigoplus_{i \in I} \gg_i$. Let $\l_i=\{(0,\ldots,0,*,0,\ldots,0)\} \iso \gg_i$ be the Lie subalgebra of $\gg$ formed by elements whose all coordinates but the $i$-th vanish. Any Lie algebra automorphism of $\gg$ permutes the $\l_i$. 
Hence, either $\Theta(\l_i)=\l_i$ and $\Theta_{|\l_i}$ is a Lie algebra involution on $\l_i$ or $\Theta(\l_i)=\l_j$ for some $i \neq j$. 
We have $\l_i \cap \gg^\Theta \neq \{0\}$ if and only if $\Theta(\l_i)=\l_i$, and so $\gg^\Theta$ determines the set of indices $I_0=\{ i \in I,\ \Theta(\l_i)=\l_i \}$. 
Moreover, for all $i,j \in I \setminus I_0$, we have $(\l_i+\l_j) \cap \gg^\Theta \neq 0$ if and only if $\Theta(\l_i)=\l_j$, and so $\gg^\Theta$ also determines the pairs of indices corresponding to the $\l_i$ that are switched by $\Theta$.  

Let $i < j$ such that $\Theta(\l_i)=\l_j$. Then 
\[(\l_i + \l_j) \cap \gg^\Theta=\{(0,\ldots,0,\varphi(z),0\ldots,0,z,0,\ldots,0)\},\] 
where $\varphi$ is some Lie algebra isomorphism $\l_j \iso \l_i$, and 
\small
\[
\begin{array}{cccc}
\Theta_{| \l_i + \l_j}  \colon & \l_i+\l_j & \to & \l_i+\l_j \\
  & (0,\ldots,0,x,0\ldots,0,y,0,\ldots,0) & \mapsto & (0,\ldots,0,\varphi(y),0\ldots,0,\varphi^{-1}(x),0,\ldots,0).
\end{array}
\]
\normalsize

So it remains only to prove that if $\Theta(\l_i)=\l_i$, then $\l_i^\Theta(=\gg^\Theta \cap \l_i)$ determines $\Theta_{|\l_i}$. Hence, we can assume that $\gg$ is simple. Let $\gg_0=\gg^\Theta$, and let $\gg_1 \subseteq \gg$ be the subspace on which $\Theta$ acts as the scalar $-1$.  We want to show that the Lie subalgebra $\gg_0$ determines the $\gg_0$-submodule $\gg_1$.

If the involution $\Theta$ is an inner automorphism of $\gg$, then $\gg_0$ contains a Cartan subalgebra $\h$ of $\gg$ (see \cite[\S~26.3]{Tim11}), and so $\gg_1$ is an $\h$-stable complement of $\gg_0$ in $\gg$. 
Since the root subspaces of $\gg$ are $1$-dimensional, the $\h$-submodule $\gg_1$ is necessarily the sum of all the root subspaces not contained in $\gg_0$. Therefore $\gg_1$ is uniquely determined by $\gg_0$, and so $\Theta$ is uniquely determined by $\gg^\Theta$.

If the involution $\Theta$ is an outer automorphism of $\gg$ (only possible for ADE type), then $\gg_1$ is an irreducible $\gg_0$-submodule of $\gg$ by \cite[Prop.~3.1]{Kac80}. If the $\gg_0$-submodule $\gg_0$ of $\gg$ does not contain a summand isomorphic to $\gg_1$, then there is a unique $\gg_0$-stable complement of $\gg_0$ in $\gg$, and so this complement must be $\gg_1$. Using the classification of symmetric spaces given in \cite[Table~26.3]{Tim11} we verify case by case that this is indeed the case.
\end{proof}

\begin{remark}\label{rk:G not reductive} \item
\begin{enumerate}[(1)]
\item 
Prop.~\ref{prop:fixed locus of an involution} is true more generally for a connected reductive algebraic group $L$ whose center $Z(L)$ has dimension at most $1$. But it is not true  if $\dim(Z(L)) \geq 2$. 
Consider  for instance $L=Z(L)=\G_m^2$, and the group involutions $\theta_1(x,y)=(y,x)$ and $\theta_2(x,y)=(x^2y^{-1},x^3y^{-2})$. Then $L^{\theta_1}=L^{\theta_2}=\{(t,t), t \in \G_m\}$ but $\theta_1 \neq \theta_2$.
\item If $\theta_1, \theta_2$ are two group involutions on a connected reductive algebraic group $L$ such that $L^{\theta_1}=L^{\theta_2}$, then we can verify that $\theta_1$ and $\theta_2$ are conjugate in $\Aut_{gr}(L)$. Hence, for connected reductive algebraic groups, the fixed locus determines the conjugacy class of a group involution.
\end{enumerate}
\end{remark}

The next result was proved by Akhiezer and Cupit-Foutou in \cite[Th.~4.4]{ACF14}, but stated for a split real group structure on $G$. It was then generalized by Snegirov in \cite[Th.~1.1]{Sne20} for quasi-split group structures over a large field of characteristic zero (note that $\R$ is such a large field), but then the proof is more technical. For sake of completeness and for the reader's convenience, we reproduce their proof in our setting.

\begin{proposition} \label{prop:HcongugateH'}
If $H \subseteq G$ is a spherical subgroup satisfying $N_G(N_G(H))=N_G(H)$ and $\sigma_{qs}(H)=gHg^{-1}$ for some $g \in G$, then there exists a subgroup $H' \subseteq G$ conjugate to $H$ such that $\sigma_{qs}(H')=H'$. 
\end{proposition}

\begin{proof}
(For this proof, and this proof only, the reader is assumed to be a bit familiar with the theory of equivariant embeddings for spherical homogeneous spaces; see \cite{Kno91} for an exposition.) 

Let $N=N_G(H)$. The condition $\sigma_{qs}(H)=gHg^{-1}$ implies that $\sigma_{qs}(N)=gNg^{-1}$. Hence, by Lem.~\ref{lem: two conditions} and Rk.~\ref{rk:N(H)=H}, the $G$-variety $Y=G/N$ has a ($G,\sigma_{qs}$)-equivariant real structure that we denote by $\mu$. 
Also, this real structure is unique since $\Aut_\C^G(Y) \iso N_G(N)/N = \{1\}$. 

By \cite[Cor.~7.2]{Kno96} the variety $Y$ admits a wonderful compactification $\overline{Y}$, which is smooth, projective (see \cite[Prop.~3.18]{Avd15}), and has a unique closed orbit $Y_0$, which is therefore a flag variety $G/P$. The colored fan of the $G$-equivariant embedding $Y \hookrightarrow \overline{Y}$ is determined by the cone ($\V,\emptyset$), where $\V$ is the valuation cone of $Y$, which is stable for the $\Gamma$-action on the set of colored cones induced by $\sigma_{qs}$ (see \cite{Hur11} or \cite{Wed18} for details on this $\Gamma$-action). Hence, the equivariant real structure $\mu$ on $Y$ extends on $\overline{Y}$ by \cite[Th.~9.1]{Wed18} (see also \cite[Th.~2.23]{Hur11}). 

The restriction $\mu_0=\mu_{|Y_0}$ is a $(G,\sigma_{qs})$-equivariant real structure on $Y_0=G/P$. Thus, by Lem.~\ref{lem: two conditions}, the parabolic subgroups $\sigma_{qs}(P)$ and $P$ are conjugate. By \cite[Prop.~3.9]{MJT18}, there exists a parabolic subgroup $P'$ conjugate to $P$ such that $\sigma_{qs}(P')=P'$ (since $\sigma_{qs}$ is quasi-split). Hence, $\mu'_0(kP')=\sigma_{qs}(k)P'$ is a $(G,\sigma_{qs})$-equivariant real structure on $Y_0$, equivalent to $\mu_0$, with a fixed point $eP'$. Thus $\mu_0$ has a fixed point in $Y_0$; in particular, $\mu$ has a fixed point in $\overline{Y}$. By \cite[Cor.~2.2.10]{Man17}, since $\overline{Y}$ is smooth, the set of $\mu$-fixed points is Zariski dense in $\overline{Y}$, and so $\mu$ has a fixed point in the open orbit $Y=G/N$.
 
Let $g_0N$ be a $\mu$-fixed point in $Y$. Let $\sigma'_{qs}=\inn_{g_0}^{-1} \circ \sigma_{qs} \circ \inn_{g_0}$, and let $\mu'$ be the $(G,\sigma'_{qs})$-equivariant real structure defined by $\mu'(kN)=g_0^{-1}\mu(g_0kN)$. Then $\mu'(eN)=eN$, and computing the stabilizers on both sides yields $\sigma'_{qs}(N)=N$. Thus 
\[ N=\sigma'_{qs}(N)=g_0^{-1}\sigma_{qs}(g_0)\sigma_{qs}(N)\sigma_{qs}(g_0^{-1}) g_0=g_0^{-1}\sigma_{qs}(g_0)gNg^{-1}\sigma_{qs}(g_0^{-1}) g_0\]
and so $g_0^{-1}\sigma_{qs}(g_0)g \in N_G(N)=N$. 
It follows  that $\sigma_{qs}(H')=H'$, where $H'=g_0Hg_{0}^{-1}$. This concludes the proof.
\end{proof}

\begin{remark}
As remarked by Avdeev in \cite{Avd13}, it is not true that $N_L(N_L(H))=N_L(H)$ for any spherical subgroup $H$ of a connected reductive algebraic group $L$. Here we give a simple counter-example pointed to us by Bart Van Steirteghem. Let 
\[L=\GL_2 \text{ and } H=\left \{\begin{bmatrix}
a & 0 \\ 0 & 1
\end{bmatrix} \text{ with } a \in \C^* \right \}.\] 
Then $N_L(H)=\left \{\begin{bmatrix}
a & 0 \\ 0 & b
\end{bmatrix} \text{ with } a,b \in \C^* \right \}$ and $N_L(N_L(H))/N_L(H)\iso \Z/2\Z$.
\end{remark}

We thank Jacopo Gandini for pointing to us that if $H \subseteq G$ is a spherical subgroup such that $N_G(H)/H$ is finite, then $N_G(N_G(H))=N_G(H)$. This is crucial in the proof of the following result.

\begin{corollary} \label{cor:H conjugate for symmetric subgroups}
Let $H$ be a symmetric subgroup of $G$ such that $\sigma_{qs}(H)$ is conjugate to $H$. 
Then there exists a subgroup $H' \subseteq G$ conjugate to $H$ such that $\sigma_{qs}(H')=H'$. 
Equivalently, there exists $\sigma'_{qs}$ conjugate to $\sigma_{qs}$ by an inner automorphism of $G$ such that $\sigma'_{qs}(H)=H$.
\end{corollary}

\begin{proof}
Since symmetric subgroups are spherical, it suffices to verify that $N_G(H)=\newline N_G(N_G(H))$ and then to apply Prop.~\ref{prop:HcongugateH'}.

Denoting $K = N_G(H^0)$, we have the inclusions $H^0 \subseteq H \subseteq N_G(H) \subseteq K$.
Then $H$ is a normal subgroup of $K$ if and only if $H/H^0$ is a normal subgroup of $K/H^0$, which is true because $H^0$ is a spherical subgroup of $G$ (since $H$ is spherical), and so $K/H^0$ is an abelian group (see \cite[\S~5.2]{BP87} or \cite[Th.~6.1]{Kno91}). Hence, $K \subseteq N_G(H)$, which yields $K=N_G(H)$. 
As $H$ is a symmetric subgroup of a semisimple algebraic group, the group $N_G(H)/H$ is finite (Prop.~\ref{prop:H abelian and finite}). Therefore $H^0$ has finite index in $K$. It follows that $K^0=H^0$, and thus $N_G(K) \subseteq N_G(K^0)=N_G(H^0)=K$.

Finally, $H'=cHc^{-1}$ satisfies $\sigma_{qs}(H')=H'$ if and only if $\sigma'_{qs}=\inn_{c}^{-1} \circ \sigma_{qs} \circ \inn_c$ satisfies $\sigma'_{qs}(H)=H$, which proves the last statement of the corollary.
\end{proof}

Before stating Th.~\ref{th:main result existence}, which is the main result of this \S~\ref{sec:existence}, we need to define the action of the Galois group $\Gamma=\Gal(\C/\R)$ on $N_G(G^\theta)/G^\theta$.

\begin{definition} \label{def:Gamma-action}
Let $\sigma=\inn_c \circ \sigma_{qs}$ be a real group structure on $G$.
If $\leftexp{\sigma}\theta \sim \theta$, then $\sigma_{qs}(G^\theta)=gG^\theta g^{-1}$ for some $g \in G$. Hence, by Cor.~\ref{cor:H conjugate for symmetric subgroups}, there exists a quasi-split real group structure $\sigma'_{qs}$, equivalent to $\sigma_{qs}$, such that $\sigma'_{qs}(G^\theta)=G^\theta$. Then $\sigma'_{qs}(N_G(G^\theta))=N_G(G^\theta)$, and so $\sigma'_{qs}$ induces a real group structure $\tau$ on $N_G(G^\theta)/G^\theta$ defined by $\tau(nG^\theta)=\sigma'_{qs}(n)G^\theta$. The $\Gamma$-action on $N_G(G^\theta)/G^\theta$ that we will consider in the following is the one given by $\tau$. (Note that this $\Gamma$-action does not depend on the choice of $\sigma'_{qs}$ in the conjugacy class of $\sigma_{qs}$ by inner automorphisms.)
\end{definition}

\begin{theorem} \label{th:main result existence} 
Let $\sigma=\inn_c \circ \sigma_{qs}$ be a real group structure on $G$.
Let $\theta$ be a group involution on $G$ and let $G^\theta \subseteq H \subseteq N_G(G^\theta)$ be a symmetric subgroup. Then the following four conditions are equivalent:
\begin{enumerate}
\item  \label{item main th i} $G/H$ has a $(G,\sigma_{qs})$-equivariant real structure;
\item  \label{item main th ii} $\leftexp{\sigma}\theta \sim \theta$ and the $\Gamma$-action on $N_G(G^\theta)/G^\theta$ of Def.~\ref{def:Gamma-action} stabilizes $H/G^\theta$; 
\item  \label{item main th iii} $H$ is conjugate to $\sigma_{qs}(H)$;
\item  \label{item main th iv} $H$ is conjugate to $\sigma(H)$.
\end{enumerate}
Moreover $G/H$ has a $(G,\sigma)$-equivariant real structure if and only if the (equivalent) conditions \ref{item main th i}-\ref{item main th iv} are satisfied and $\Delta_H(\sigma)=0$ with $\Delta_H$ the map defined by \eqref{eq:map Delta}.
\end{theorem}

\begin{proof}
The equivalence of \ref{item main th i} and \ref{item main th iii} follows from Lem.~\ref{lem: two conditions} and Cor.~\ref{cor:H conjugate for symmetric subgroups}. Indeed, if $H$ is conjugate to $\sigma_{qs}(H)$, then we can find $H'$ conjugate to $H$ such that $\sigma_{qs}(H')=H'$. By Rk.~\ref{rk:conjugate}, we can replace $H$ by $H'$ and then the two conditions of Lem.~\ref{lem: two conditions} are satisfied with $g=1$.

The equivalence of \ref{item main th iii} and \ref{item main th iv} follows from the fact that $\sigma=\inn_c \circ \sigma_{qs}$.

We now prove the equivalence of \ref{item main th ii} and \ref{item main th iii}. 
By Prop.~\ref{prop:H abelian and finite} the group $G^\theta$ is connected and $N_G(G^\theta)^0=H^0=G^\theta$. 
Also, by Rk.~\ref{rk:conjugate} and Cor.~\ref{cor:H conjugate for symmetric subgroups} we can replace  \ref{item main th iii} by the condition \ref{item main th iii}' given by  $\sigma_{qs}(H)=H$. (This boils down to conjugate $\theta$ by some inner automorphism.) Then
\begin{align*}
  \sigma_{qs}(H^0)=H^0 \Leftrightarrow \sigma(H^0)=cH^0c^{-1} &\Leftrightarrow \ \sigma(G^\theta)=cG^\theta c^{-1}\\ 
                                                                                   &\Leftrightarrow \ G^{\,^\sigma\hskip-1pt \theta}=G^{\,^{{\rm inn}_c}\theta
} \\
                                                                          & \Leftrightarrow \ \leftexp{\sigma}{\theta}=\leftexp{\inn_c}{\theta}  \ \Leftrightarrow \ \leftexp{\sigma}{\theta} \sim \theta,  
\end{align*}
where the penultimate equivalence comes from Prop.~\ref{prop:fixed locus of an involution}. 
Also,  $\sigma_{qs}(G^\theta)=G^\theta$ implies that $\sigma_{qs}(N_G(G^\theta))=N_G(G^\theta)$, and so the $\Gamma$-action on $N_G(G^{\theta})/G^{\theta}$ stabilizes $H/G^\theta$ if and only if $\sigma_{qs}(H)=H$. This finishes to prove the equivalence \ref{item main th ii} $\Leftrightarrow$ \ref{item main th iii}.

Finally, the last claim of the theorem follows from Prop.~\ref{prop:coho condition}.
\end{proof}

\begin{remark} \label{rk:cyclic}
If $N_G(G^\theta)/G^\theta$ is a cyclic group, then the $\Gamma$-action stabilizes each subgroup of $N_G(G^\theta)/G^\theta$, which simplifies the condition \ref{item main th ii} in Th.~\ref{th:main result existence}. 
\end{remark}

\begin{corollary}
Let $\sigma_{qs}$ and $\theta$ be as in Th.~\ref{th:main result existence}. Then $G/G^\theta$ has a $(G,\sigma_{qs})$-equivariant real structure if and only if $G/N_G(G^\theta)$ does, and this is the case if and only if $\leftexp{\sigma}\theta\sim \theta$.
\end{corollary}

\begin{proof}
If $H=G^\theta$ or $H=N_G(G^\theta)$, then the $\Gamma$-action on $N_G(G^\theta)/G^\theta$ trivially stabilizes $H/G^\theta$, and so the result follows from the equivalence of \ref{item main th i} and \ref{item main th ii} in Th.~\ref{th:main result existence}.
\end{proof}

\begin{example} \label{ex3}
Let $G=\SL_n \times \SL_n$ with $n$ odd and $n \geq 3$, let $\sigma\colon(g_1,g_2)\mapsto (\overline{g_2},\overline{g_1})$, and let $\theta\colon(g_1,g_2)\mapsto (\leftexp{t}{g}_{1}^{-1},\leftexp{t}{g}_{2}^{-1})$. Then $\leftexp{\sigma}\theta=\theta$ and $N_G(G^\theta)/G^\theta \iso \Z/n\Z \times \Z/n\Z$ on which $\Gamma$ acts by $\gamma \cdot(a,b)=({b}^{-1},a^{-1})$.
Thus, since $\sigma$ is quasi-split (it preserves the usual Borel subgroup), it follows from Th.~\ref{th:main result existence} that there exists a $(G,\sigma)$-equivariant real structure on the symmetric space $G/H$ if and only if $(a,b) \in H/G^\theta$ implies $(b,a) \in H/G^\theta$, that is, $H/G^\theta$ is stable under the operation of exchanging  the two factors of $N_G(G^\theta)/G^\theta$.
\end{example}

\begin{example} \label{ex4}
Let $n \geq 2$ and let $G=\SL_{2n}$.
Let $\sigma$ be a real group structure on $G$ obtained by an inner twist of $\sigma_{qs}$, where $\sigma_{qs}$ is the quasi-split real group structure defined in Example~\ref{ex1}, let $\theta$ be the group involution defined in Example~\ref{ex2}, and let $G^\theta \subseteq H \subseteq N_G(G^\theta)$. Then $\leftexp{\sigma_{qs}}{\theta}=\theta$ (and so $\leftexp{\sigma}{\theta} \sim \theta$), and the $\Gamma$-action on $N_G(G^\theta)/G^\theta \iso \Z/n\Z$ stabilizes $H/G^\theta$ by Rk.~\ref{rk:cyclic}. Hence, by Th.~\ref{th:main result existence}, the symmetric space $G/H$ has a ($G,\sigma_{qs}$)-equivariant real structure. 

It remains to compute $\Delta_H(\sigma)$ to determine whether $G/H$ has a ($G,\sigma$)-equivariant real structure. Let $S=\{0,\ldots,n\}$. The equivalence classes of the real group structures on $G$ obtained as an inner twist of $\sigma_{qs}$ are in bijection with $S$. For $s \in S$, we denote by $\sigma_s$ the real group structure whose real locus $G(\C)^{\sigma_s}$ is $\SU(n+s,n-s,\R)$. 
Borovoi determined in \cite[Appendix~A, Table~2]{MJT18} that $H^2(\Gamma,Z(G)) \iso Z(G)/2Z(G) \iso \Z/2\Z$ and that $\delta(\sigma_s)= s \mod 2$. Let $\xi$ be a primitive $2n$-th root of unity. Then $H=\left\langle \xi^r I_{2n} ,G^\theta \right\rangle$, for some positive integer $r$ dividing $2n$, and $A:=N_G(H)/H \iso Z(G)/ (Z(G) \cap H) \iso \Z/t\Z$ with $t=\gcd(r,n)$. We verify that the $\Gamma$-action on $Z(G)$ (and so also on $A$) is trivial, thus 
\[H^2(\Gamma,A) \iso A/2A \iso \left\{
    \begin{array}{ll}
       \Z/2\Z& \mbox{ if $t$ is even}; \\
       \{0\} &  \mbox{ if $t$ is odd}.
    \end{array}
\right.\] 
The map $\chi_H^* : H^2(\Gamma,Z(G)) \iso Z(G)/2Z(G) \to H^2(\Gamma,A)\iso A/2A$ defined in \S~\ref{subsec1} is the map induced by the quotient map $Z(G) \to A \iso Z(G)/(Z(G) \cap H)$, hence it is the identity map if $t$ is even (resp. the trivial map if $t$ is odd). It follows that $\Delta_H(\sigma_s)=0$ if and only if $s$ is even or $t$ is odd. Therefore, $G/H$ has a ($G,\sigma_s$)-equivariant real structure if and only if $s$ is even or $t$ is odd. 
\end{example}

\section{Number of equivalence classes} \label{sec:number of eq classes}
Let $\sigma$ be a real group structure on $G$, and let $X=G/H$ be a symmetric space.
\textbf{We suppose that there exists a ($G,\sigma$)-equivariant real structure $\mu$ on $X$.} 
Then $\mu$ determines a $\Gamma$-action on $A=\Aut_\C^G(X)\iso N_G(H)/H$; indeed, the generator $\gamma$ acts on $A$ by $\mu$-conjugation. 

\begin{notation}
In this section, and contrary to the previous examples, we will follow the usual conventions and use the multiplicative notation for the group law in $A$, even if $A$ is a finite abelian group in our case by Prop.~\ref{prop:H abelian and finite}.
\end{notation}

\begin{definition} \label{def:Galois H1}
If $A$ is a $\Gamma$-group, then  the first Galois cohomology pointed set is $H^1(\Gamma,A)=Z^1(\Gamma,A)/\sim$, where $Z^1(\Gamma,A)=\{ a \in A \ | \   a^{-1}= \ga a \}$ and two elements $a_1$, $a _2 \in Z^1(\Gamma,A)$ satisfy $a_1 \sim a_2$ if $a_2=b^{-1} a_1\ga  b$ for some $b \in A$.
\end{definition}

\begin{remark}\label{rk:2 torsion}
If $A$ is an abelian group, then $H^1(\Gamma, A)$ is an abelian group.
Moreover, $a^2 =a (a^{-1})^{-1}=a(\ga a)^{-1} \sim 1$ for all $a \in Z^1(\Gamma,A)$. In the case where $H^1(\Gamma, A)$ is finite, this implies that its cardinal is a power of $2$.
\end{remark}

By \cite[Lem.~2.11]{MJT18} the set of equivalence classes of $(G,\sigma)$-equivariant real structures on $X$ is in bijection with the set $H^1(\Gamma, A)$, hence our goal in this section is to determine the cardinal of $H^1(\Gamma, A)$. 

Before stating the next result, we need some extra notation. Let $\Gamma'=\{e,\gamma'\} \iso \Z/2\Z$ acting on $A$ by $\leftexp{\gamma'}a=\leftexp{\gamma}{a}^{-1}$. (This $\Gamma'$-action is well-defined since $A$ is abelian.) For $p$ a prime number, let $A_p$ be the maximal $p$-subgroup of $A$.

\begin{proposition} \label{prop:number of structures}
We suppose that $X=G/H$ has a $(G,\sigma)$-equivariant real structure $\mu$, and we consider the actions of $\Gamma$ and $\Gamma'$ on $A$ defined above.
\begin{enumerate}
\item \label{item i} There exists $n \geq 0$ such that $H^1(\Gamma,A)\iso H^1(\Gamma,A_2)\iso (\Z/2\Z)^n$. In particular, there are $2^n$ equivalence classes of $(G,\sigma)$-equivariant real structures on $G/H$.
\item \label{item ii} The integer $n$ can be calculated explicitly as follows: $|A_2^{\Gamma'}|\cdot |A_2^\Gamma|/|A_2|=2^n$.
\end{enumerate}
\end{proposition}

\begin{proof}
\ref{item i}: In our situation, the group $A$ is a finite abelian group (Prop.~\ref{prop:H abelian and finite}). Hence, $A$ is isomorphic to a finite product of abelian $p$-groups  $A \iso \prod_p A_p$, and each $A_p$ is $\Gamma$-stable since $\Gamma=\Gal(\C/\R)$ acts on $A$ by group involution. Thus $H^1(\Gamma,A)=\prod_p H^1(\Gamma,A_p)$. But each $H^1(\Gamma,A_p)$ is itself an abelian $p$-group (by definition of the Galois cohomology), and since every element of $H^1(\Gamma,A)$ is $2$-torsion (Rk.~\ref{rk:2 torsion}), we see that $H^1(\Gamma,A_p)=\{1\}$ if $p \neq 2$. Therefore 
\[H^1(\Gamma,A) \iso H^1(\Gamma,A_2) \iso (\Z/2\Z)^n \text{ for some } n \geq 0.\]

\noindent \ref{item ii}: In order to calculate $n$ (or, more precisely, $2^n$), we consider certain subgroups of $A_2$. Let us note that $Z:=Z^1(\Gamma,A_2)=A_2^{\Gamma'}$ is a subgroup of $A_2$, and $H^1(\Gamma,A)= H^1(\Gamma, A_2)=Z/B$, where $B$ is the subgroup of $Z$ defined by $B=\{a\cdot \leftexp{\gamma}a^{-1};\, a \in A_2\}$. The map $\varphi:A_2\to B$ given by $\varphi(a)=a\cdot \leftexp{\gamma}a^{-1}$ is a surjective group homomorphism (since $A_2$ is abelian). The kernel is exactly $A_2^\Gamma$. Thus, the cardinality of $H^1(\Gamma,A_2)$ is given by $|A_2^{\Gamma'}|/|B|$, and $|B|=|A_2|/|A_2^\Gamma|$. This proves the result.
\end{proof}

\begin{remark}
It is easy to  give an upper-bound for $n$. Suppose  that $A_2$ is a product of $r$ cyclic groups. Then $Z=Z^1(\Gamma,A_2)$ is a  subgroup of $A_2$, and therefore a product of $r'$ cyclic groups, where $r'\le r$. In particular, $H^1(\Gamma,A_2)$ is  a quotient of the group $Z/Z^2\iso (\Z/2\Z)^{r'}$. This shows that $n\le r'\le r$. 
\end{remark}

\begin{corollary}\label{cor:number-cyclic} 
Suppose that $X=G/H$ has a $(G,\sigma)$-equivariant real structure and that $A=N_G(H)/H$ is cyclic of order $m$. If $m$ is odd, then the real structure is unique up to equivalence, and if $m$ is even, there are exactly $2$ inequivalent  real structures on $X$. 
\end{corollary}

\begin{proof} If $m$ is odd, then $A_2$ is trivial, and the result holds. If $m$ is even, then $A_2$ is cyclic of order at least two. There are two possible $\Gamma$-actions on $A_2$. Either the action is trivial, or $\leftexp{\gamma}{a}=a^{-1}$ for all $a \in A$, in which case the $\Gamma'$-action is trivial. In both cases, since $A_2$ has a unique element of order two, the result holds. 
\end{proof}

\begin{example}

We saw that $G/H$ has a ($G,\sigma_s$)-equiva\-riant real structure if and only if $s$ is even or $t=\gcd(r,n)$ is odd. Moreover, the group $N_G(H)/H$ is cyclic, so Corollary \ref{cor:number-cyclic} applies. We find that the the number of equivalence classes of $(G,\sigma_s)$-equivariant real structures on $G/H$ is given by
\[|A_2^{\Gamma'}|\cdot |A_2^\Gamma|/|A_2|=|A_2^{\Gamma'}|=|\{a \in A | a^2=1\}|=\left\{
    \begin{array}{ll}
        1 & \text{ if $t$ is odd;}  \\
       2 & \text{ if $s$ and $t$ are even.}
    \end{array}
\right.\] 
\end{example}

\bigskip

\noindent \textbf{Acknowledgments.}
We are very grateful to Michael Bulois, Jacopo Gandini, and Bart Van Steirteghem for interesting exchanges related to this work. We also thank the anonymous referee for his/her helpful comments.

\def\cprime{$'$}

\end{document}